\documentclass[11pt]{article}
\usepackage{amsmath,amsthm,mathrsfs}
\usepackage[english]{babel}

\usepackage{amsfonts,amstext,amssymb,verbatim,epsfig}
\usepackage{pgf,tikz}
\usepackage{float}
\usetikzlibrary{arrows}
\usepackage{hyperref}
\usepackage{cite}
\usepackage{epstopdf}
\usepackage{color}
\usepackage{transparent}
\usepackage{subfigure}

\usepackage{xcolor}
\usepackage[normalem]{ulem}

\hypersetup{colorlinks=true,pdfborder=000,  citecolor  = blue, citebordercolor= {magenta}}
\hypersetup{linkcolor=magenta}
\makeatletter
\let\reftagform@=\tagform@
\def\tagform@#1{\maketag@@@{(\ignorespaces\textcolor{purple}{#1}\unskip\@@italiccorr)}}
\renewcommand{\eqref}[1]{\textup{\reftagform@{\ref{#1}}}}
\makeatother

\makeatletter

\DeclareUrlCommand\ULurl@@{%
  \def\UrlLeft{\uline\bgroup}%
  \def\UrlRight{\egroup}}
\def\ULurl@#1{\hyper@linkurl{\ULurl@@{#1}}{#1}}
\DeclareRobustCommand*\ULurl{\hyper@normalise\ULurl@}
\makeatother

\def\lessim{\ \lower4pt\hbox{$
		\buildrel{\displaystyle <}\over\sim$}\ }
\def\gessim{\ \lower4pt\hbox{$\buildrel{\displaystyle >}
		\over\sim$}\ }

\newcommand{\e}{\mathbb{E}}
\newcommand{\p}{\mathbb{P}}

\newcommand{\indi}{\ensuremath{\boldsymbol 1}}

\newtheorem{lemma}{\bf Lemma}

\newtheorem{theorem}{\bf Theorem}

\newtheorem*{prediction}{\bf Prediction}

\newtheorem{proposition}{\bf Proposition}

\theoremstyle{remark}
\newtheorem{remark}{\bf Remark}

\newcommand{\8}{\infty}

\newcommand{\rz}{\mathbb{R}}
\newcommand{\px}{\mathcal{P}}
\newcommand{\al}{\alpha}

\newcommand{\sign}{{\rm sign}}
\newcommand{\ux}{\mathcal{U}}
\newcommand{\ga}{\gamma}

\newenvironment{Proof of lemma}{\noindent{\bf Proof of Lemma}}{\hfill$\Box$\newline}
\newenvironment{Proof of theorem}{\noindent{\bf Proof of Theorem}}{\hfill{\footnotesize${\square}$}\newline}
\newenvironment{Proof of theorems}{\noindent{\bf Proof of Theorems}}{\hfill$\Box$\newline}
\newenvironment{Proof of proposition}{\noindent{\bf Proof of Proposition}}{\hfill$\Box$\newline}
\newenvironment{Proof of propositions}{\noindent{\bf Proof of Propositions}}{\hfill$\Box$\newline}
\newenvironment{Proof of exercise}{\noindent{\it Proof of Exercise:}}{\hfill$\Box$}
\begin{document}
\title{The SK model is infinite step replica symmetry breaking \\
	at zero temperature}
\author{Antonio Auffinger \thanks{Department of Mathematics, Northwestern University, tuca@northwestern.edu, research partially supported by NSF Grant CAREER DMS-1653552 and NSF Grant DMS-1517894.} \\
\small{Northwestern University}\and Wei-Kuo Chen \thanks{School of Mathematics, University of Minnesota. Email: wkchen@umn.edu, research partially supported by NSF Grant DMS-1642207, NSF Grant Career DMS-1752184, and Hong Kong Research Grants Council GRF-14302515.}
\\
\small{University of Minnesota}
\and Qiang Zeng \thanks{Department of Mathematics, Northwestern University, qzeng.math@gmail.com.}\\\
\small{Northwestern University}
}
\date{February 4, 2020}








\maketitle   





\begin{abstract} We prove that the Parisi measure of the mixed $p$-spin model at zero temperature has infinitely many points in its support. This establishes  Parisi's  prediction  that the functional order parameter of the Sherrington-Kirkpatrick model is not a step function at zero temperature.
	As a consequence,  we show that the number of levels of broken replica symmetry in the Parisi formula of the free energy  diverges as the temperature goes to zero.
\end{abstract}

\section{Introduction and main results}

\subsection{The Model and Parisi's Solution}
The study of  mean field spin glass models is a very rich and important part of theoretical physics \cite{Montanaribook, MPV, Parisibook}. For mathematicians, it is a challenging program \cite{T03,T11,P13}. Roughly speaking, the main goal is to study the global maxima or, more generally, the ``largest individuals" of a stochastic process with ``high-dimensional" correlation structure.

The classic example of such a process is the mixed $p$-spin model. Its Hamiltonian (or energy) $H_N$ is defined on the spin configuration space $\Sigma_N=\{-1,1\}^N$ by
\begin{align*}
H_N(\sigma)&=X_N(\sigma)+h\sum_{i=1}^N\sigma_i
\end{align*}
for
$$X_N(\sigma):=\sum_{p\geq 2}c_pX_{N,p}(\sigma).$$
Here, the processes $X_{N,p}$'s are pure $p$-spin Hamiltonians defined as
\begin{align*}
X_{N,p}(\sigma)=\frac{1}{N^{(p-1)/2}}\sum_{1\leq i_1,\ldots,i_p\leq N}g_{i_1,\ldots,i_p}\sigma_{i_1}\cdots\sigma_{i_p},
\end{align*}
where $g_{i_1,\ldots,i_p}$ are i.i.d. standard normal variables for all $1\leq i_1,\ldots,i_p\leq N$ and $p\geq 2.$ The real sequence $(c_p)_{p\geq 2}$ satisfies $\sum_{p\geq 2}2^pc_p^2<\infty.$
The parameter $h\in\mathbb{R}$ denotes the strength of the external field. By definition, the covariance of $X_N$ can be computed as
\begin{equation*}
\mathbb E X_{N}(\sigma^1)X_N(\sigma^2)=N\xi(R_{1,2}),
\end{equation*}
where
\begin{equation*}
\xi(s):=\sum_{p\geq 2}c_p^2s^p
\end{equation*}
and $$R_{1,2}=R(\sigma^{1}, \sigma^{2}):= \frac{1}{N}\sum_{i=1}^N\sigma_i^1\sigma_i^2$$ is the normalized inner product between $\sigma^1$ and $\sigma^2$, known as the overlap.
The covariance structure of $X_N$ is as rich as the structure of the metric space  $(\Sigma_N, d)$, where $d$ is the normalized Hamming distance on $\Sigma_{N}$,
$$
d(\sigma^{1},\sigma^{2}) = \frac{1-R(\sigma^{1},\sigma^{2})}{2}.
$$

The problem of computing the maximum energy (or the ground state energy) of $H_N$ as $N$ diverges is a rather nontrivial task.  Standard statistical mechanics deals with this problem by considering  the Gibbs measure
\begin{equation*}
G_{N,\beta} (\sigma) = \frac{1}{Z_{N,\beta}} e^{\beta H_N(\sigma)}
\end{equation*} and the free energy
\begin{align*}
F_{N,\beta}=\frac{1}{\beta N}\log Z_{N,\beta},
\end{align*}
where $Z_{N,\beta}$ is the partition function of $H_N$ defined as
\begin{align*}
Z_{N,\beta} = \sum_{\sigma\in \Sigma_{N}} e^{ \beta H_N(\sigma)}.
\end{align*}
The parameter $\beta=1/(kT)>0$ is called the inverse temperature, where $k$ is the Boltzmann constant and $T$ is the absolute temperature. The main goal in this approach is to try to describe the large $N$ limit of the sequences of the free energies $F_{N,\beta}$ and the Gibbs measures $G_{N,\beta}$. When the temperature $T$ decreases, large values of $H_N$ become more important (to both the partition function $Z_{N,\beta}$ and to the Gibbs measure $G_{N,\beta}$) and they prevail over the more numerous smaller values.  Since $H_{N}$ is a high-dimensional correlated field with a large number of points near its global maximum, this question becomes very challenging, especially for small values of $T$.

When $\xi(s)=s^{2}/2$ and $h=0$, the model above is the famous Sherrington-Kirkpatrick (SK) model introduced in \cite{SK}, as a mean field modification of the Edwards-Anderson model \cite{EA}. Using a non-rigorous replica trick and a replica symmetric hypothesis, Sherrington and Kirkpatrick \cite{SK} proposed a solution to the limiting free energy of the SK model. Their solution however was incomplete; an alternative solution was proposed  in 1979 in a series of ground-breaking articles by Giorgio Parisi \cite{Pa79,Parisi, Pa80,Pa83}, where it was foreseen that:
\begin{enumerate}
\item[$(i)$] The limiting free energy is given by a variational principle, known as the Parisi formula,
\item[$(ii)$] The Gibbs measures  are asymptotically ultrametric,
\item[$(iii)$] At low  enough temperature, the  symmetry of replicas is broken infinitely many times.
\end{enumerate}

The first two predictions were confirmed in the past decade. Following the beautiful discovery of Guerra's broken replica symmetry scheme \cite{Guerra}, the Parisi formula was proved in the seminal work of Talagrand \cite{Talagrand} in 2006 under the convexity assumption of $\xi$. Later, in 2012, the ultrametricity conjecture was established by Panchenko \cite{Panch} assuming the validity of the extended Ghirlanda-Guerra identities \cite{GG}. These identities are known to be valid for the SK model with an asymptotically vanishing perturbation and for generic models  without any perturbation. Here, the model is said to be generic if the span of $\{1\}\cup \{s^p: c_p\neq 0\}$ is dense in the space of continuous functions $C[-1,1]$ under the maximum norm. As a consequence of ultrametricity, the Parisi formula was further extended to all mixed $p$-spin models by Panchenko \cite{P05} utilizing the Aizenman-Sims-Starr scheme \cite{ASS}.

More precisely, the Parisi formula for general mixtures is stated as follows. Denote by $\mathcal{M}$ the collection of all cumulative distribution functions $\alpha$ on $[0,1]$ and by $\alpha(d s)$ the probability induced by $\alpha$. For $\alpha\in\mathcal{M}$, define
\begin{align}\label{pf}
\mathcal{P}_{\beta}(\alpha)=\frac{\log 2}{\beta}+\Psi_{\alpha,\beta}(0,h)-\frac{1}{2}\int_0^1\beta\alpha(s)s\xi''(s)d s,
\end{align}
where $\Psi_{\alpha,\beta}(t,x)$ is the weak solution to the following nonlinear parabolic PDE,
\begin{align*}
\partial_t\Psi_{\alpha,\beta}(t,x)&=-\frac{\xi''(t)}{2}\bigl(\partial_{xx}\Psi_{\alpha,\beta}(t,x)+\beta\alpha(t)(\partial_x\Psi_{\alpha,\beta}(t,x))^2\bigr)
\end{align*}
for $(t,x)\in[0,1)\times\mathbb{R}$ with boundary condition
$$
\Psi_{\alpha,\beta}(1,x)=\frac{\log \cosh \beta x}{\beta}.
$$
For the existence and regularity of $\Psi_{\alpha,\beta}$, we refer the readers to \cite{ParisiMeasure,JT2}. The  Parisi formula \cite{Talagrand} states that
\begin{align}\label{eq0}
F_\beta:=\lim_{N\rightarrow\infty}F_{N,\beta}&=\inf_{\alpha\in\mathcal{M}}\mathcal{P}_\beta(\alpha) \,\,\,\,a.s.
\end{align}
The infinite dimensional variational problem on the right-hand side of \eqref{eq0} has a unique minimizer \cite{ParisiMeasure}, denoted by $\alpha_{P,\beta}$. The measure $\alpha_{P,\beta}(d t)$ induced by $\alpha_{P,\beta}$ is known as the Parisi measure \cite{MPV}\footnote{The Parisi measure is the inverse of the functional order parameter $q(x)$ in \cite{Parisi}, sometimes written as $x(q)$.}. In Parisi's solution, it is predicted that the Parisi measure is the limiting distribution of the overlap $R(\sigma^1,\sigma^2)$ under the measure $\mathbb E G_N^{\otimes 2}$. More importantly, this together with the asymptotic ultrametricity implies that the spin configurations under the Gibbs measure form a hierarchical clustering structure and the number of levels therein is determined by the number of points in the support of the Parisi measure $\alpha_{P,\beta}(dt)$. The importance of $(iii)$ lies on the fact that it indicates the phase transition of the model between high and low temperature regimes. While at high temperature the clusters contain no layers, at low temperature they begin to possess clustering structures with multiple layers. The statement $(iii)$ advocates the existence of infinitely many layers within the clusters at low enough temperature.  In a nutshell, the Parisi measure is the key ingredient of the matter that describes the structure of the Gibbs measure as well as the free energy of the system. See \cite{MPV,P13} for detailed discussion.

The importance of the Parisi measure leads to the following classification. If a Parisi measure $\alpha_{P,\beta}(d t)$ is a Dirac measure, we say that the model is replica symmetric (RS). For $k\geq 1$, we say that  the model has $k$ levels of replica symmetry breaking ($k$-RSB) if the Parisi measure is atomic and has exactly $k+1$ jumps. If the Parisi measure is neither RS nor $k$-RSB for some $k\geq 1,$ then the model has infinite levels of replica symmetry breaking ($\infty$-RSB).  We will also say that the model is at least $k$-RSB if the Parisi measure contains at least $k+1$ distinct values in its support.

It is expected that when $\infty$-RSB happens, the functional order parameter has a non-empty interval in its support. This prediction in the physics literature is named as full-step replica symmetry breaking (FRSB). This prediction plays an inevitable role in Parisi's original solution of the SK model. It can be written as:

\begin{prediction}[Parisi]\label{con:ds}
For any $\xi$ and $h$, there exists a critical inverse temperature $\beta_{c}>0$ such that for any $\beta > \beta_{c}$,  the mixed $p$-spin model is FRSB.
\end{prediction}


\vspace{0.1cm}

\subsection{Main Results}
In this paper, we establish this prediction at zero temperature. To prepare for the statement of our main result, we recall the Parisi formula for the ground state energy of $H_N$ as follows. First of all, the Parisi formula allows us to compute the ground state energy of the model by sending the temperature $T$ to zero,
\begin{align}
\label{eq1}
GSE := \lim_{N\to \8} \max_{\sigma \in \Sigma_N} \frac{H_N(\sigma)}N=\lim_{\beta \rightarrow\infty}F_{\beta}=\lim_{\beta \rightarrow\infty}\inf_{\alpha\in\mathcal{M}}\mathcal{P}_\beta(\alpha),
\end{align}
where the validity of the first equality can be found, for instance, in Panchenko's book \cite[Chapter 1]{P13}. In physics literature, it is a convention that the ground state energy is defined as the minimum of the Hamiltonian.  This is indeed equivalent to our formulation here since the disorder coefficients $g_{i_1,\ldots,i_p}$'s are symmetric with respect to the origin. Recently, the analysis of the $\beta$-limit of the second equality was carried out in Auffinger-Chen \cite{AC16} and it was discovered that the ground state energy can be written as a Parisi-type formula. Let $\ux$ denote the collection of all cumulative distribution functions $\gamma$ on $[0,1)$ induced by any measures on $[0,1)$ and satisfying $\int_0^1\ga(t) d t<\8$. Denote by $\gamma(d t)$ the measure that induces $\gamma$ and endow $\ux$ with the $L^1(d t)$-distance. For each $\ga\in \mathcal{U}$, consider the weak solution to the Parisi PDE,
\begin{align*}
\partial_t \Psi_\ga(t,x) = -\frac{\xi''(t)}2 \bigl(\partial_{xx} \Psi_\ga(t,x)+\ga(t) (\partial_x \Psi_\ga(t,x))^2\bigr)
\end{align*}
for $(t,x)\in[0,1)\times \rz$ with boundary condition
\[
\Psi_\ga(1,x) = |x|.
\]
One may find the existence and regularity properties of this PDE solution in \cite{CHL16}. The Parisi functional at zero temperature is given by
\begin{equation*}
\px(\ga) = \Psi_\ga(0,h) -\frac12 \int_0^1 t\xi''(t)\ga(t) d t.
\end{equation*}
Auffinger and Chen \cite{AC16} proved that the maximum energy can be computed through
\begin{equation}\label{eq:par_fml}
GSE=\inf_{\ga\in \ux} \px(\ga)\quad a.s.
\end{equation}
We call this variational representation the Parisi formula at zero temperature. It was proved in \cite{CHL16} that this formula has a unique minimizer which is denoted by $\ga_P.$ We call $\ga_P(d t)$ the Parisi measure at zero temperature. We say that the model is $\infty$-RSB at zero temperature if $\ga_P(d t)$ contains infinitely many points in its support. Our first main result is a proof of Parisi's $\infty$-RSB prediction at zero temperature. I


\begin{theorem}\label{th:main}
For any $\xi$ and $h,$ the mixed $p$-spin model at zero temperature is $\infty$-RSB. 
\end{theorem}

\begin{remark}
The above theorem is a first step towards the validation of Parisi's FRSB Prediction. At zero temperature (and at sufficiently low temperatures), it is expected that the support of the Parisi measure contains an interval. This remains as an important open question.
\end{remark}

Similar to the role of the Parisi measure at positive temperature played in describing the behavior of the model, the Parisi measure at zero temperature also has its own relevance in understanding the energy landscape of the Hamiltonian around the maximum energy.
\begin{proposition}\label{cor1}
For any $\xi$ and $h,$ if $u$ lies in the support of $\gamma_P(d t)$, then for any $\varepsilon,\eta>0$, there exists some constant $K>0$ independent of $N$ such that
\begin{align}\label{eq4}
\p\Bigl(\exists\sigma^1,\sigma^2\,\,\mbox{such that}\,\,R_{1,2}\in (u-\varepsilon,u+\varepsilon)\,\,\mbox{and}\,\,\frac{H_N(\sigma^1)}{N},\frac{H_N(\sigma^2)}{N}\geq GSE-\eta\Bigr)\geq 1- Ke^{-\frac{N}{K}}
\end{align}
for all $N\geq 1.$
\end{proposition}
This means that for any $u\in\mbox{supp } \gamma_P$,  one can always find two spin configurations around the maximum energy such that their overlap is near $u$ with overwhelming probability. Recently, it was proved in \cite{CHL16} that when $h=0$ and the model is even, i.e., $c_p=0$ for all odd $p$, the system exhibits the so-called multiple peaks feature in the sense that there exist exponentially many near maximizers and they are nearly orthogonal to each other. Theorem \ref{th:main} and Proposition \ref{cor1} together add a more fruitful structure to the landscape of the Hamiltonian. In view of Proposition \ref{cor1}, it would be of great interest to understand the size of the number of the pairs $(\sigma^1,\sigma^2)$ with $R_{1,2}\in (u-\varepsilon,u+\varepsilon)$ and $H_N(\sigma^1),H_N(\sigma^2)\geq N(GSE-\eta)$ for a typical realization.

\begin{remark}
	\rm The relationship between the Parisi solution and combinatorial optimization problems has been investigated since the 1980s, see \cite{MPV} for references. The problem of computing the maximum energy is also generally known as the Dean's problem and is frequently used to motivate the theory of mean field spin glasses, see \cite{P13,MPV}. More recently, the formula \eqref{eq1} at zero temperature has appeared in other optimization problems related to theoretical computer science such as extremal cuts on sparse random graphs, see \cite{DMS,Sen} and the references therein. In the physics literature, the FRSB prediction is also discussed in \cite{MontanariP,OS1,OS2}. \end{remark}

\begin{remark}
	\rm For the mixed $p$-spin model defined on the sphere, $\{\sigma\in \mathbb{R}^N: \sum_{i=1}^N\sigma_i^2=N\}$, the corresponding Parisi-type formula at both positive and zero temperatures is much simpler than \eqref{pf} and \eqref{eq1} as there is no PDE involved. In this setting, it is known in \cite{Tal06, AChen13, AC17,AZ,JT17} that the Parisi variational formulas can be explicitly solved for some choices of the mixture parameters $(c_p)_{p\geq 2}$ and algebraic conditions on $\xi$ were presented so that the model has Parisi measures of RS, 1RS, 2RSB, and FRSB. The study of the energy landscapes also has appeared in \cite{AB, ABC, AC17,JT17,Subag,Subag2,Subag3}.

Theorem 1 and Proposition 1 indicate that the spin configurations around the maximum energy are not simply clustered into equidistant groups. This is in sharp contrast to the energy landscape of the spherical version of the mixed $p$-spin model, where in the pure $p$-spin model, i.e., $\xi(t) = t^p$ for $p \ge 3$, it was shown by Subag \cite{Subag} that around the maximum energy, the spin
configurations are essentially orthogonally structured. This structure was also presented in more general mixtures of the spherical model in the recent work of Auffinger and Chen \cite{AC17}.
\end{remark}

We now return to the positive temperature case. Recall the Parisi measure $\alpha_{P,\beta}$ introduced in \eqref{eq0}. Our second main result, as a consequence of Theorem \ref{th:main}, shows that for any mixture parameter $\xi$ and external field $h$, the number of levels of replica symmetry breaking must diverge as $\beta$ goes to infinity.

\begin{theorem}\label{thm:finiteTemp} Let $k\geq 1.$ For any $\xi$ and $h$, there exists $\beta_{k}$ such that the mixed $p$-spin model is at least $k$-RSB for all $\beta > \beta_{k}.$\end{theorem}

For the SK model without external field, $\xi(s)=s^2/2$ and $h=0,$ Aizenman, Lebowitz, and Ruelle \cite{ALR} showed that the model is RS, i.e.~the limiting distribution of $R_{1,2}$ is a Dirac mass if $\beta$ is in the high temperature regime $\beta<1$. Later, it was also understood by Toninelli \cite{Toni} that the model is not RS in the low temperature region $\beta>1$. It is conjectured that the whole region $\beta >1$ is expected to be of FRSB. One may find recent progress on the phase transition from RS to RSB for more general mixtures in \cite{AChen13,JT171,Toni}. However, the existence of FRSB in the mixed $p$-spin models remains unknown.
\vspace{0.1cm}

\subsection{Our Approach}
	The main novelty of our approach to Theorem \ref{th:main} is to explore the Parisi formula for the ground state energy \eqref{eq:par_fml} by considering a  perturbation of the Parisi functional around the point $1$. In short, we show that it is always possible to lower the value of the Parisi functional of any atomic measure with finite atoms by adding a large enough jump near $1$. The argument is quite delicate, and depends on some observations that may look counterintuitive at first glance; see Remark \ref{intuitive} below. At finite temperature, since the Parisi measure is a probability measure, the idea of adding a large jump is not feasible.  
	Theorem \ref{thm:finiteTemp}  follows from Theorem \ref{th:main} after some weak convergence considerations. Recently there have been many papers investigating the properties of the Parisi functional via the stochastic optimal control theory for the Parisi PDE, e.g., \cite{ParisiMeasure,JT17,JT171}. As our analysis here is quite subtle and relies on many identities related to the Gaussian expectations as well as some technical calculations, we prefer adapting the setting and framework as in Talagrand's book \cite{T11} in order to keep track of our control with clarity.

\section{Lowering the value of the Parisi functional}

In this section, we show that for any atomic $\gamma(d s)$ with finitely many jumps, one can always lower the value of the Parisi functional by a perturbation of $\gamma$ around $1$.
Let $\gamma\in \mathcal{U}$ be fixed. Suppose that $\ga(dt)$ is atomic and consists of finitely many jumps, that is,
\begin{align*}
\ga(t) = \sum_{i=0}^{n-1} m_{i} \indi_{[q_{i}, q_{i+1})} (t)+ m_n \indi_{[q_n,1)}(t),
\end{align*}
where $(q_i)_{0\leq i\leq n}$ and $(m_i)_{0\leq i\leq n}$ satisfy
\begin{align}
\begin{split}
\label{eq2}
&0=q_0< q_1<q_2<\cdots<q_n<1,\\
&0\leq m_0<m_1<m_2<\cdots<m_n<\infty.
\end{split}
\end{align}
Here and in what follows, $\indi_B(t)=\indi_{[t\in B]}$ is the indicator function of the set $B\subset \rz$. Let $m_{n+1}$ be any number greater than $m_n.$ For any $q\in (q_n,1),$ consider the following perturbation of $\ga$
\begin{align}\label{pert}
\ga_q(t) = \sum_{i=0}^{n-1} m_{i} \indi_{[q_{i},q_{i+1})}(t) + m_n\indi_{[q_{n}, q)} (t)+ m_{n+1} \indi_{[q, 1)} (t),
\end{align}
where $q_i$'s and $m_i$'s are from \eqref{eq2}.
In other words, we add a jump to the top of $\gamma.$
Our main result is the following theorem. It says that if $m_{n+1}$ is large enough, then the Parisi functional evaluated at perturbed measure $\gamma_q(dt)$ has a smaller value than $\mathcal{P}(\gamma)$ locally for $q$ near $1$.

\begin{theorem}\label{thm3}
	There exist $m_{n+1}>m_n$ and $\eta\in (q_n,1)$ such that
	\begin{align*}
	\mathcal{P}(\gamma_q)<\mathcal{P}(\gamma)
	\end{align*}
	for all $\eta\leq q<1$.	
\end{theorem}

The following three subsections are devoted to the proof of Theorem \ref{thm3}.

\subsection{Probabilistic representation of $\mathcal P$}
We start by observing that the Parisi functional at $\ga$ admits a probabilistic expression by an application of the Cole-Hopf transformation to the Parisi PDE. Indeed, let $z_0,\ldots,z_n$ be i.i.d. standard Gaussian random variables. For $q_i,i=0,...,n,$ given in \eqref{eq2}, we denote
\begin{align*}
J&=h+\sum_{i=0}^{n-1}z_i\sqrt{\xi'(q_{i+1})-\xi'(q_i)}+z_{n}\sqrt{\xi'(1)-\xi'(q_n)}.
\end{align*}
Set
\begin{align*}
X_{n+1}&=|J|.
\end{align*}
Define iteratively, for $m_i, i=0,..., n,$ given in \eqref{eq2},
\begin{align*}
X_i=\frac{1}{m_i}\log \e_{z_i}\exp m_i X_{i+1}.
\end{align*}
where $\e_{z_i}$ stands for the expectation for $z_i.$ Here $X_0$ is defined as $\e_{z_0} X_1$ if $m_0=0.$ Then by applying the Cole-Hopf transformation, $\Psi_{\ga}(0,h)=X_0$, and thus
	\begin{align*}
	\px(\ga) &= X_0-\frac12\sum_{i=0}^{n-1} m_i \int_{q_i}^{q_{i+1}} t\xi''(t) d t-\frac{m_n}2\int_{q_n}^1 t\xi''(t) d t.
	\end{align*}
Recall the perturbation $\gamma_q$ from \eqref{pert}. Clearly $\gamma_q=\gamma$ on $[0,q)$ for all $q_{n}< q < 1.$ For notational convenience, we denote
\begin{align}\label{eq3}
q_{n+1}=q,\,\,q_{n+2}=1.
\end{align}
In a similar manner, we can express $\Psi_{\gamma_q}(0,h)$ as follows. Let $z_{n+1}$ be a standard Gaussian random variables independent of $z_1,\ldots,z_n.$ Define
\begin{align*}
Y_{n+2}&=\Bigl|h+\sum_{j=0}^{n+1}z_j\sqrt{\xi'(q_{j+1})-\xi'(q_j)}\Bigr|
\end{align*}
and iteratively, for $0\leq i\leq n+1,$
\begin{align}
\label{e:yidef}
Y_{i}&=\frac{1}{m_{i}}\log \e_{z_i} \exp m_{i} Y_{i+1}.
\end{align}
Here again we let $Y_0=\e_{z_0}Y_1$ whenever $m_0=0.$ Thus, $\Psi_{\gamma_q}(0,h)=Y_0$ for any $q\in (q_{n},1).$ As a result,
	\begin{align}
	\label{e:pxgaq}
	\px(\ga_q) &= Y_0-\frac12\sum_{i=0}^{n-1} m_i \int_{q_i}^{q_{i+1}} t\xi''(t) d t-\frac{m_n}{2}\int_{q_n}^qt\xi''(t)d t-\frac{m_{n+1}}2\int_{q}^1 t\xi''(t) d t.
	\end{align}
In particular, we have $\lim_{q\to1-}\Psi_{\gamma_q}(0,h)=\Psi_{\gamma}(0,h)$ and $\lim_{q\to1-}\px(\ga_q)=\mathcal{P}(\gamma).$

\subsection{Some auxiliary lemmas}\label{sub2.1}

We state two propositions that will be heavily used in our main proof in the next subsection.
Let $0\leq a<t<b$ and $0< m<m'.$
Denote by $z$ a standard normal random variable. Define
\begin{align}
\begin{split}\label{add:eq2}
A (t,x)&=\frac{1}{m'}\log \e\exp m'\bigl|x+z\sqrt{b-t}\bigr|,\\
B(t,x)&=\frac{1}{m}\log \e\exp m A (t,x+z\sqrt{t-a}),\\
C(t,x)&=\e \bigl(\partial_xA(t,x+z\sqrt{t-a})\bigr)^2V(t,x,x+z\sqrt{t-a}),
\end{split}
\end{align}
where
\begin{align*}
V(t,x,y)&=e^{m(A (t,y) -B(t,x))}.
\end{align*}
For any $(t,x)\in [a,b)\times\mathbb{R}$, define random variables,
\begin{align*}
\tilde{V}(t,x)&=V(t,x,x +z \sqrt{t-a}),\\
A_t(t,x)&=\partial_tA(t,x+z\sqrt{t-a}),\\
A_{tx}(t,x)&=\partial_{tx}A(t,x+z\sqrt{t-a}),\\
A_{x}(t,x)&=\partial_xA(t,x+z\sqrt{t-a}),\\
A_{xx}(t,x)&=\partial_{xx}A(t,x+z\sqrt{t-a}),\\
A_{xxx}(t,x)&=\partial_{xxx}A(t,x+z\sqrt{t-a}).
\end{align*}
We stress that $A_t(t,x)\neq \partial_t A(t,x), A_x(t,x)\neq \partial_x A(t,x)$, etc.~in our notation. Note that $\e\tilde V(t,x)=1$. The main results of this subsection are the following two propositions.

\begin{proposition}
	\label{add:prop1}
	For any $(t,x)\in[a,b)\times\mathbb{R}$, we have that
	\begin{align}
	\begin{split}
	\label{add:prop1:eq1}
	\partial_tB(t,x)&=\frac{(m-m')}{2}C(t,x)
	\end{split}
	\end{align}
	and
	\begin{align}
	\begin{split}
	\label{add:prop1:eq2}
	\partial_tC(t,x)&=\e  \bigl(A _{xx}(t,x)^2+2(m -m')A _{xx}(t,x)A _x(t,x)^2\bigr)\tilde V(t,x)\\
	&+\frac{(m -m')m }{2}\bigl(\e  A _x(t,x)^4\tilde V(t,x)-\bigl(\e A _x(t,x)^2\tilde V(t,x)\bigr)^2\bigr).
	\end{split}
	\end{align}
\end{proposition}
\begin{remark}
	The functions \eqref{add:eq2} and the formula \eqref{add:prop1:eq1} also appeared in \cite[Section 14.7]{T11} in a similar manner, where in the exponent of $A$, the author used the random variable $\beta^{-1}\log \cosh(\beta (x+z\sqrt{b-t}))$ instead of $|x+z\sqrt{b-t}|.$
\end{remark}

\begin{proposition}
	\label{add:prop2} For $(t,x)\in [a,b)\times\mathbb{R},$
	we have that
	\begin{align}
	\begin{split}
	\label{add:prop2:eq1}
	\lim_{t\rightarrow b-}C(t,x)&=1
	\end{split}
	\end{align}
	and
	\begin{align*}
	\liminf_{t\rightarrow b-}\partial_tC(t,x)&\geq\frac{2(m+m')}{3}\Delta(x),
	\end{align*}
	where
	\begin{align}\label{add:prop2:eq3}
	\Delta(x)&=\frac{2}{\sqrt{2\pi(b-a)}}\frac{e^{-\frac{x^2}{2(b-a)}}}{\e e^{m |x+z \sqrt{b-a}|}}.
	\end{align}
\end{proposition}

Before we turn to the proof of Propositions \ref{add:prop1} and \ref{add:prop2}, we first gather some fundamental properties of the function $A .$

\begin{lemma}
	\label{add:lem0} $A $ is the classical solution to the following PDE with boundary condition $A (b,x)=|x|,$
	\begin{align}\label{PDE}
	\partial_tA(t,x)&=-\frac{1}{2}\bigl(\partial_{xx}A(t,x)+m'\bigl(\partial_xA(t,x)\bigr)^2\bigr)
	\end{align}
	for $(t,x)\in[a,b)\times\mathbb{R}.$ In addition,
	\begin{align}
	\begin{split}\label{add:lem0:eq1}
	&|\partial_xA(t,x)|\leq 1,\,\,(t,x)\in [a,b)\times\mathbb{R},
	\end{split}\\
	\begin{split}\label{add:lem0:eq2}
	&\lim_{t\rightarrow b-}\partial_xA(t,x)=\mbox{\rm sign}(x),\,\,\forall x\in \mathbb{R}\setminus \{0\},
	\end{split}\\
	\begin{split}
	\label{add:lem0:eq3}
	&\lim_{t\rightarrow b-}\e A_x(t,x)^{2k}\tilde V(t,x)=1,\,\,\forall (t,x)\in [a,b)\times\mathbb{R},\;\forall \; k\in \mathbb{N}, \forall \; 0<m<m',
	\end{split}
	\end{align}
	where $\mbox{sign}(x)=1$ if $x>0$ and $=-1$ if $x<0.$
\end{lemma}

\begin{proof}
	Define $$
	g(t,x)=e^{\frac{(b-t){m'}^2}{2}+m'x}\Phi \Bigl(m'\sqrt{b-t}+\frac{x}{\sqrt{b-t}}\Bigr),
	$$
	where $\Phi$ is the cumulative distribution function of a standard normal random variable. Note that a direct computation gives
	\begin{align*}
	\e e^{m'|x+z\sqrt{b-t}|}&=g(t,x)+g(t,-x).
	\end{align*}
	Thus,
	\begin{align*}
	A (t,x)=\frac{1}{m'}\log\bigl(g(t,x)+g(t,-x)\bigr).
	\end{align*}
	From this expression, we can compute that
	\begin{align}
	\begin{split}\label{add:lem0:proof:eq1}
	\partial_xA(t,x)&=\frac{g(t,x)-g(t,-x)}{g(t,x)+g(t,-x)},\\
	\partial_{xx}A(t,x)
	&=m'\Bigl(1-\Bigl(\frac{g(t,x)-g(t,-x)}{g(t,x)+g(t,-x)}\Bigr)^2\Bigr)+2\Gamma(t,x),\\
	\partial_tA(t,x)&=-\frac{m'}{2}-\Gamma(t,x),
	\end{split}
	\end{align}
	where
	\begin{align*}
	\Gamma(t,x):=\frac{1}{\sqrt{2\pi(b-t)}}\frac{e^{-\frac{x^2}{2(b-t)}}}{g(t,x)+g(t,-x)}.
	\end{align*}
	Therefore, these equations together validate \eqref{PDE}. From the first equation, we can also conclude \eqref{add:lem0:eq1} and \eqref{add:lem0:eq2}.	
	Note that $\lim_{t\rightarrow b-}V(t,x,y)=V(b,x,y)=\frac{e^{m|y|}}{\e e^{m|x+z\sqrt{b-a}|}}$ and $\ln V(t,\cdot,\cdot)$ is at most of linear growth. From \eqref{add:lem0:eq1} and \eqref{add:lem0:eq2}, the dominated convergence theorem implies \eqref{add:lem0:eq3}.
	
\end{proof}

\begin{proof}[\bf Proof of Proposition \ref{add:prop1}] To lighten our notation, we will drop the dependence on $(t,x)$ in this proof. Recall that the Gaussian integration by parts states that for a standard normal random variable $z$, $\e zf(z)=\e f'(z)$ for all absolutely continuous functions $f$ satisfying that $\ln|f|$ is at most of linear growth at infinity; see e.g. \cite[Section A.4]{T11}.
	From this formula and the PDE \eqref{PDE},
	\begin{align*}
	\partial_tB&=\e\bigl(A _t+\frac{z}{2\sqrt{t-a}}A _x\bigr)\tilde V\\
	&=\e\bigl(-\frac{1}{2}\bigl(A _{xx} + m' A _x ^2\bigr) +\frac{1}{2}\bigl(A _{xx} + m A _x ^2\bigr)\bigr)\tilde V\\
	&=\frac{m -m'}{2}\e A _{x} ^2\tilde V,
	\end{align*}
	which gives \eqref{add:prop1:eq1}.
	To compute the partial derivative of $C$ in $t$, write $\partial_tC=I+II$
	for
	\begin{align*}
	I&:=2\e  \Bigl(A _{tx}+\frac{z }{2\sqrt{t-a}}A _{xx}\Bigr)A _x\tilde V\\
	II&:=m \e  A _x^2\Bigl(A _t+\frac{z }{2\sqrt{t-a}}A _x-\partial_tB\Bigr)\tilde V.
	\end{align*}
	Here, from \eqref{PDE}, since
	\begin{align*}
	\partial_{tx}A&=-\frac{1}{2}\bigl(\partial_{xxx}A+2m'\partial_{xx}A \cdot\partial_xA\bigr)
	\end{align*}
	for any $(t,x)\in [a,b)\times\mathbb{R},$
	using the Gaussian integration by parts again gives
	\begin{align*}
	I&=2\e  \Bigl(A _{tx}A _x+\frac{1}{2}\bigl(A _{xxx}A _x+A _{xx}^2+m A _{xx}A _x^2\bigr)\Bigr)\tilde V\\
	&=\e  \Bigl(-A _x\bigl(A _{xxx}+2m'A _{xx}A _x\bigr)+\bigl(A _{xxx}A _x+A _{xx}^2+m A _{xx}A _x^2\bigr)\Bigr)\tilde V\\
	&=\e  \bigl(A _{xx}^2+(m -2m')A _{xx}A _x^2\bigr)\tilde V.
	\end{align*}
	In addition, from \eqref{PDE} and the Gaussian integration by parts,
	\begin{align*}
	II&=m \e  \Bigl(-\frac{1}{2}\bigl(A _{xx}A _x^2+m'A _x^4\bigr)+\frac{1}{2}\bigl(3A _{xx}A _x^2+m A _x^4\bigr)-A _x^2\partial_tB\Bigr)\tilde V\\
	&=m \e  A _{xx}A _x^2\tilde V+\frac{(m -m')m }{2}\bigl(\e  A _x^4\tilde V-\bigl(\e A _x^2\tilde V\bigr)^2\bigr).
	\end{align*}
	From these, \eqref{add:prop1:eq2} follows.\end{proof}

To handle the limits in Proposition \ref{add:prop2}, we need two lemmas.

\begin{lemma}\label{add:lem1}
	For any odd $k\geq 1,$ there exists a constant $K$ independent of $t$ such that
	\begin{align}
	\label{add:lem1:eq1}
	\e A _{x}(t,x) ^{k-1}A _{xx}(t,x) \tilde V(t,x)\leq \frac{Ke^{m |x|}}{\sqrt{t-a}}
	\end{align}
	for all $t\in[a,b)$ and $x\in \mathbb{R}$. Moreover,
	\begin{align}\label{add:lem1:eq2}
	\lim_{t\rightarrow b-}\e A _{x}(t,x) ^{k-1}A _{xx}(t,x)\tilde V(t,x) &=\frac{1}{k}\Delta(x),
	\end{align}
	where $\Delta(x)$ is defined in Proposition \ref{add:prop2}.
\end{lemma}

\begin{proof}
	Define
	\begin{align*}
	D(t,x)&=\e  z A _x ^{k}(t,x)\tilde V(t,x) .
	\end{align*}
	Note that $|\partial_xA (t,x)|\leq 1$ and $B(t,x) \geq 0$. We have
	\begin{align*}
	V(t,x,y)&=e^{m (A (t,y)-B(t,x))}\leq e^{mA (t,0)+m|y|}.
	\end{align*}
	Using the Gaussian integration by parts, we can write
	\begin{align}\label{add:lem1:proof:eq1}
	D(t,x)&=\sqrt{t-a}\e \bigr(kA _{x}(t,x) ^{k-1}A _{xx}(t,x) +m A _{x}(t,x)^{k+1}\bigl)\tilde V(t,x) .
	\end{align}
	This and the previous inequality together imply \eqref{add:lem1:eq1} since
	\begin{align*}
	k\sqrt{t-a}\e  A _{x}(t,x) ^{k-1}A _{xx}(t,x)\tilde V(t,x) &\leq D(t,x)\leq e^{m A (a,0)+m |x|}\e |z |e^{m |z |\sqrt{b-a}},
	\end{align*}
	where the first inequality used the fact that $k+1$ is even.
	Next, we verify \eqref{add:lem1:eq2}. Note that from \eqref{add:lem0:eq1} and \eqref{add:lem0:eq2}, the dominated convergence theorem implies
	\begin{align*}
	\lim_{t\rightarrow b-}D(t,x)&=\frac{\e  z \mbox{sign}(x+z \sqrt{b-a})e^{m |x+z \sqrt{b-a}|}}{\e e^{m |x+z \sqrt{b-a}|}}=\sqrt{b-a}\bigl(\Delta(x)+m \bigr),
	\end{align*}
	where the second equation used the fact that
	\begin{align}\label{add:lem4}
	\e z \mbox{sign}(x+z \sqrt{b-a})e^{m |x+z \sqrt{b-a}|}&=\frac{2}{\sqrt{2\pi}}e^{-\frac{x^2}{2(b-a)}}+m \sqrt{b-a}\e  e^{m |x+z \sqrt{b-a}|}.
	\end{align}
	See the verification of this equation in Lemma \ref{lem1} in the appendix.
	In addition, since $k+1$ is even, \eqref{add:lem0:eq3} yields
	\begin{align*}
	\lim_{t\rightarrow b-}\e A _{x}(t,x)^{k+1}\tilde V(t,x) &=1.
	\end{align*}
	Thus, from \eqref{add:lem1:proof:eq1} and the last two limits,
	\begin{align*}
	\Delta(x)\sqrt{b-a}+m \sqrt{b-a}&=\lim_{t\rightarrow b-}D(t,x)=\sqrt{b-a}\bigl(k\lim_{t\rightarrow b-}\e A _{x}(t,x) ^{k-1}A _{xx}(t,x)\tilde V(t,x) +m \bigr),
	\end{align*}
	from which \eqref{add:lem1:eq2} follows.
\end{proof}

\begin{lemma}
	\label{add:lem3} We have that
	\begin{align*}
	\liminf_{t\rightarrow b-}\e A _{xx}(t,x)^2\tilde V(t,x)&\geq \frac{4m'}{3}\Delta(x).
	\end{align*}
\end{lemma}

\begin{proof}
	Recall the middle equation of \eqref{add:lem0:proof:eq1}. We see that
	\begin{align}
	\label{add:lem2:eq1}
	A _{xx}(t,x)&=m'(1-A _x(t,x)^2)+2\tilde\Gamma(t,x),
	\end{align}
	on $[a,b)\times\mathbb{R},$ where $\tilde\Gamma(t,x):=\Gamma(t, x+z\sqrt{b-a})$. Using \eqref{add:lem0:eq3}, \eqref{add:lem2:eq1}, and Lemma \ref{add:lem1} with $k=1$ gives
	\begin{align*}
	\lim_{t\rightarrow b-}\e\tilde\Gamma(t,x) \tilde V(t,x)&=\frac12\Delta(x).
	\end{align*}
	Also multiplying both sides of \eqref{add:lem2:eq1} by $A _x(t,x)^2$ and applying \eqref{add:lem0:eq3} and Lemma \ref{add:lem1} with $k=3$ yield
	\begin{align*}
	\lim_{t\rightarrow b-}\e A_{x}(t,x)^{2}\tilde\Gamma(t,x)\tilde  V(t,x)=	    	 \frac{1}{6}\Delta(x).
	\end{align*}
	From \eqref{add:lem2:eq1}, since
	\begin{align*}
	A _{xx}(t,x)^2&=\bigl(m'(1-A _x(t,x)^2)+2\tilde\Gamma(t,x)\bigr)^2\\
	&\geq [m'(1-A _x(t,x)^2)]^2+4m' (1-A _x(t,x)^2)\tilde\Gamma(t,x),
	\end{align*}
	the announced result follows by the last two limits.    	 	
\end{proof}

\begin{proof}[\bf Proof of Proposition \ref{add:prop2}]
	The statement \eqref{add:prop2:eq1} follows from \eqref{add:eq2} and \eqref{add:lem0:eq3}. From \eqref{add:prop1:eq2}, \eqref{add:lem0:eq3}, Lemma \ref{add:lem3}, and \eqref{add:lem1:eq2}, we find
	\begin{align*}
	\liminf_{t\rightarrow b-}\partial_tC(t,x)&\geq \liminf_{t\rightarrow b-}\e \bigl(A _{xx}(t,x)^2+2(m-m')A _{xx}(t,x)A _x(t,x)^2\bigr)\tilde V(t,x)\\
	&\geq \liminf_{t\rightarrow b-}\e A _{xx}(t,x)^2\tilde V(t,x)+2(m-m')\lim_{t\rightarrow b-}\e A _{xx}(t,x)A _x(t,x)^2 \tilde V(t,x)\\
	&=\frac{4m'}{3}\Delta(x)+\frac{2(m-m')}{3}\Delta(x)\\
	&=\frac{2(m+m')}{3}\Delta(x),
	\end{align*}
	which completes the proof.
\end{proof}

\subsection{Proof of Theorem \ref{thm3}}

Recall the sequences $(q_i)_{0\leq i\leq n+2}$ and $(m_i)_{0\leq i\leq n+1}$ from \eqref{eq2} and \eqref{eq3}. Recall the quantities $m,m',a,b$ and the functions $A,B,C,V$ from \eqref{add:eq2}. From now on, we take
\begin{align*}
m&=m_n,\,\,m'=m_{n+1},\\
a&=\xi'(q_n),\,\,b=\xi'(1),
\end{align*}
and let
\begin{align*}
\hat{A}(q,x)&=A(\xi'(q),x),\\
\hat{B}(q,x)&=B(\xi'(q),x),\\
\hat{C}(q,x)&=C(\xi'(q),x),\\
\hat{V}(q,x,y)&=V(\xi'(q),x,y).
\end{align*}
For $0\leq i\leq n$, set
\[
W_i=\exp m_i(Y_{i+1}-Y_i),
\]
where $Y_i$'s are given in \eqref{e:yidef}. Denote
\begin{align*}
Z&=h+\sum_{j=0}^{n-1}z_j\sqrt{\xi'(q_{j+1})-\xi'(q_j)},
\end{align*}
and
	\begin{align*}
	\phi(q)&=\e W_0\cdots W_{n-1}\hat{C}(q,Z).
	\end{align*}

\begin{lemma}\label{lem2}
	We have that
	\begin{align}
	\label{lem2:eq1}
	\partial_q\mathcal{P}(\gamma_q)=\frac{\xi''(q)}{2}(m_{n+1}-m_n)\bigl(q-\phi(q)\bigr)
	\end{align}
	and
	\begin{align}
	\begin{split}\label{lem2:eq2}
	\phi'(q)&=\frac{\xi''(q)(m_{n}-m_{n+1})}{2}\sum_{i=0}^{n-1}m_i\e\Bigl[ W_0\cdots W_i\bigl(\e_{i+1}\bigl[W_{i+1}\cdots W_{n-1}\hat{C}(q,Z)\bigr]\bigr)^2\Bigr]\\
	&-\frac{\xi''(q)(m_{n}-m_{n+1})}{2}\sum_{i=0}^{n-1}m_i\e \Bigl[W_0\cdots W_{i}\e_{z_i}\bigl[W_i\bigl(\e_{i+1}\bigl[W_{i+1}\cdots W_{n-1}\hat{C}(q,Z)\bigr]\bigr)^2\bigr]\Bigr]\\
	&+\e W_0\cdots W_{n-1}\hat{C}_q(q,Z),
	\end{split}
	\end{align}
	where  $\e_i$ is the expectation with respect to $z_i,\ldots,z_{n-1}$ and $\hat{C}_q$ is the partial derivative with respect to $q.$
\end{lemma}

\begin{proof}
	Observe that for $0\leq i\leq n-1$,
	\begin{align*}
	\partial_qY_i&=\e_{z_i} W_i \partial_q Y_{i+1}.
	\end{align*}
	An induction argument yields
	\begin{align*}
	\partial_qY_i&=\e_i W_i\cdots W_{n-1}\partial_qY_{n}
	\end{align*}
	for $0\leq i\leq n-1$.
	Since $Y_{n}=\hat{B}(q,Z),$
	the equation \eqref{add:prop1:eq1} leads to
	\begin{align}\label{add:eq3}
	\partial_qY_i&=\frac{\xi''(q)(m_{n}-m_{n+1})}{2}\e_i W_i\cdots W_{n-1}\hat{C}(q,Z).
	\end{align}
	From \eqref{e:pxgaq}, since
	\begin{align*}
	\partial_q\mathcal{P}(\gamma_q)&=\partial_qY_0+\frac{q \xi''(q)}{2}(m_{n+1}-m_n),
	\end{align*}
    this and \eqref{add:eq3} with $i=0$ yield \eqref{lem2:eq1}. On the other hand, for $0\leq i\leq n-1,$ from \eqref{add:eq3},
	\begin{align*}
	\partial_qW_i&=m_i\bigl(\partial_qY_{i+1}-\partial_qY_i\bigr)W_i\\
	&=\frac{\xi''(q)}{2}(m_{n}-m_{n+1})m_iW_i\bigl(\e_{i+1}W_{i+1}\cdots W_{n-1}\hat{C}(q,Z)-\e_iW_i\cdots W_{n-1}\hat{C}(q,Z)\bigr)
	\end{align*}
    where $\e_{i+1}W_{i+1}\cdots W_{n-1}\hat{C}(q,Z)=\hat{C}(q,Z)$ if $i=n-1.$
	Finally, since
	\begin{align*}
	\phi'(q)&=\sum_{i=0}^{n-1}\e W_0\cdots W_{i-1}(\partial_qW_i)W_{i+1}\cdots W_{n-1}\hat{C}(q,Z)+\e W_0\cdots W_{n-1}\hat{C}_q(q,Z),
	\end{align*}
	plugging the last equation into this derivative yields \eqref{lem2:eq2}.
\end{proof}

\begin{proof}[\bf Proof of Theorem \ref{thm3}]
Recall $\phi'(q)$ from \eqref{lem2:eq2}. Let $\tilde W_0,\ldots, \tilde W_{n-1}$ be $W_0,\ldots, W_{n-1}$ evaluated at $q=1.$ Note that $\e_{z_i}W_i=1$ for all $0\leq i\leq n-1$, $\lim_{q\to 1-}\hat C(q,Z)=1$ by \eqref{add:prop2:eq1} and $|\hat{C}(q,Z)|\leq 1$ by \eqref{add:lem0:eq1}. Applying Fatou's lemma and conditional expectation yields that the first two lines of \eqref{lem2:eq2} cancel each other and as a result of Proposition \ref{add:prop2},
\begin{align}
\begin{split}\label{thm3:proof:eq1}
\liminf_{q\rightarrow 1-}\phi'(q)&=\liminf_{q\rightarrow 1-}\e W_0\cdots W_{n-1}\hat{C}_q(q,Z)\\
&\geq \e \tilde W_0 \cdots \tilde W_{n-1} \liminf_{q\rightarrow 1-}\hat{C}_q(q,Z)\\
&\geq \frac{2\xi''(1)(m_n+m_{n+1})}{3}\e \tilde W_0 \cdots \tilde W_{n-1} \Delta(Z),
\end{split}
\end{align}
where $\Delta(Z)$ is defined through \eqref{add:prop2:eq3} with $a=\xi'(q_n),$ $b=\xi'(1)$, and $m=m_{n}.$ We emphasize that although we do not know whether $\hat{C}_q$ is nonnegative (see \eqref{add:prop1:eq2}), the use of Fatou's lemma remains justifiable. Indeed, note that $|\hat{A}_x|\leq 1$, $\e_{z_n}\hat{V}=1$, and by \eqref{add:lem1:eq1},
\begin{align*}
0&\leq \e_{z_n}\hat{A}_{xx}(q,Z)\hat{A}_{x}^2(q,Z)\hat{V}(q,Z,Z+z_n\sqrt{\xi'(q)-\xi'(q_n)})\leq \frac{Ke^{m_{n} |Z|}}{\sqrt{\xi'(q)-\xi'(q_n)}},
\end{align*}
where $K$ is a constant independent of $q.$ From \eqref{add:prop1:eq2}, $$
\hat{C}_q(q,Z)\geq -(m_{n+1}-m_{n})\xi''(q)\Bigl(\frac{2Ke^{m_{n} |Z|}}{\sqrt{\xi'(q)-\xi'(q_n)}}+\frac{m_n}{2}\Bigr).
$$
In addition, it can be shown that each $\ln\bigl( W_0W_1\cdots W_{n-1}\bigr)$ is at most of linear growth in $z_0,\ldots,z_{n-1}$ following from the fact that each $Y_i$ is uniformly Lipschitz in the variable $z_i$ for all $q\in [q_n,1].$ This and the last inequality together validate \eqref{thm3:proof:eq1}.

Next, from \eqref{thm3:proof:eq1}, we can choose $m_{n+1}$ large enough in the beginning such that
$$
\liminf_{q\rightarrow 1-}\phi'(q)>1.
$$
Note that $\lim_{q\rightarrow 1-}\phi(q)=1$. From \eqref{lem2:eq1}, the above inequality implies that $\partial_q\mathcal{P}(\gamma_q)>0$ for our choice of $m_{n+1}$ as long as $q$ is sufficiently close to $1.$ This completes our proof.

\end{proof}

\begin{remark}\label{intuitive}
The validity of \eqref{thm3:proof:eq1} and Theorem \ref{thm3} relies on the positive lower bound of $\partial_tC$ coming from Proposition \ref{add:prop2}. When one looks at \eqref{add:prop1:eq2} together with the fact $\lim_{t\to b-}(\partial_xA)^2=1$, it is tempting to think that $\partial_tC$ is actually negative since $m'=m_{n+1}$ is taken to be large. As a result, Proposition \ref{add:prop2} may look counterintuitive. The remedy for this puzzle is the fact that $\partial_{xx}A(t,x)$ is singular in the limit $t\to b-$ and the dominated convergence theorem does not apply. These ``singular expectations'' are one of the major difficulties to prove Theorem~\ref{thm3}.  They are handled by the exact computations coming from Lemmas \ref{add:lem1} and \ref{add:lem3}.
\end{remark}

\begin{remark}
In the  work \cite{Toni}, Toninelli proved that in the SK model, the Parisi formula for the free energy possesses RSB Parisi measure when the temperature stays below the de Almeida-Thouless transition line. In view of the argument therein, he considered the first derivative of the Parisi functional $\mathcal{P}_\beta(\gamma_q)$ in the variable $m_{n+1}$ and then sent $m_{n+1}\to 1$ for $n=0$ to obtain a quantity $K(\beta,h,q)$. He then concluded the proof by calculating the first and second derivatives of $K(\beta,h,q)$ w.r.t.~$q$ near $\bar q$, where $\bar q$ is the unique solution of $\bar q = \int d\mu(z) \tanh^2(z\beta\sqrt{\bar q}+\beta h)$ for $z\thicksim N(0,1)$. In our situation, the maximum value of $m_{n+1}$ is unclear since $\ga$ is not a probability measure. Our argument considers the first and second derivatives of the perturbed Parisi functional $\px(\ga_q)$ in $q$ and evaluated at $q=1$, which will not work for the positive temperature case because  $\alpha_{P,\beta}([0,x])$ is less than or equal to one for  $x$ being close to $1$.
\end{remark}

\section{Proofs of main results}

\begin{proof}[\bf Proof of Theorem \ref{th:main}]
We prove Theorem \ref{th:main} by contradiction. First, note that it is known by \cite[Theorem 6]{CHL16} that the Parisi measure $\ga_P$ is not constantly zero. Suppose that the support of $\gamma_P$ consists of only $n\geq 1$ points. Then from Theorem \ref{thm3}, we can lower the value of the Parisi functional by a perturbation of $\gamma_P$ at $1$ defined in \eqref{pert}. This leads to a contradiction of the minimality of $\mathcal{P}(\gamma_P).$ Hence, the support of $\gamma_P$ must contain infinitely many points.
\end{proof}

\begin{remark}
The statement of Theorem \ref{th:main} can be strengthened to the fact that the Parisi measure $\ga_P$ cannot be ``flat'' near 1, i.e., $\ga_P(t)<\ga_P(1-)$ for any $0<t<1$. In fact, if this is not true, then $\gamma_P$ is a constant function on $[a,1)$ for some $a.$ One can then apply essentially the same argument as Proposition~\ref{thm3} to lower the Parisi functional. The only difference is that since $\gamma_P$ is not necessarily a step function on $[0,a),$ the term $W_1\cdots W_{n-1}$ in Lemma \ref{lem2} have to be replaced by a continuous modification using the optimal stochastic control representation for $\Psi_\gamma$ in \cite{CHL16}. We omit the details of the argument.
\end{remark}

\begin{remark}\rm
	Our argument of Theorem \ref{th:main} does not rely on the uniqueness of the Parisi measure. All we need is the existence of \emph{a} Parisi measure which was proved in \cite{AC16}.
\end{remark}

\begin{proof}[\bf Proof of Proposition \ref{cor1}] For any $u\in [-1,1]$ and $\varepsilon>0,$ consider the coupled free energy and the maximum energy
	\begin{align*}
	{C\!F}_{N,\beta}(u,\varepsilon):=\frac{1}{\beta N}\log \sum_{|R_{1,2}-u|<\varepsilon}e^{\beta H_N(\sigma^1)+\beta H_N(\sigma^2)}
	\end{align*}
	and
	\begin{align*}
	{C\!M}_N(u,\varepsilon):=\frac{1}{N}\max_{|R_{1,2}-u|<\varepsilon}\bigl(H_N(\sigma^1)+H_N(\sigma^2)\bigr).
	\end{align*}
	First we claim that if $u$ is in the support of the Parisi measure $\alpha_{P,\beta}(dt)$, then for any $\varepsilon>0,$
	\begin{align}\label{e:ecfbt} \lim_{N\rightarrow\infty}\e{C\!F}_{N,\beta}(u,\varepsilon)=2F_\beta.
	\end{align}
	Note that the expectation of the free energy $F_{N,\beta}$ and the coupled free free energy $C\!F_{N,\beta}$ are Lipschitz functions in the mixture parameters $(c_p)_{p\geq 2}$ with respect to the $\ell_2(\mathbb{R})$-norm. In addition, the Parisi measure is continuous in $\beta$ as well as the mixture parameters $(c_p)_{p\geq 2}$ (see, e.g., \cite{AC161}). From these, we may assume without loss of generality that $c_p>0$ for all $p\geq 2.$ This condition guarantees that the Parisi measure $\alpha_{P,\beta}$ is the limiting distribution of the overlap $R_{1,2}$ between two i.i.d.~samplings $\sigma^1$ and $\sigma^2$ from $G_{N},$ see \cite{P13,Tal06c}. Namely, we have
\[
\lim_{N\to+\8} \e G_N^{\otimes2}(R_{1,2}\in B) = \al_{P,\beta}(B)
\]
for any Borel set $B\subset [0,1]$. To see why the claim \eqref{e:ecfbt} holds, observe that if there exists a positive $\delta$ such that
	\begin{align*}
	\limsup_{N\rightarrow\infty}\e C\!F_{N,\beta}(u,\varepsilon)\leq 2F_\beta-\delta,
	\end{align*}
	it follows by the Gaussian concentration of measure that there exists a constant $K>0$ such that for any $N\geq 1,$ with probability at least $1-Ke^{-N/K}$,
	\begin{align*}
	C\!F_{N,\beta}(u,\varepsilon)\leq 2 F_{N,\beta}-\frac{\delta}{2}.
	\end{align*}
	Multiplying $N$ and taking exponential lead to
	\begin{align*}
	\e G_{N}^{\otimes 2}\bigl(R_{1,2}\in (u-\varepsilon,u+\varepsilon)\bigr)\leq Ke^{-N/K}+e^{-N\beta \delta/2},
	\end{align*}
	which implies, by the weak convergence of $R_{1,2}$,
	\begin{align*}
	\int_{(u-\varepsilon,u+\varepsilon)} \alpha_{P,\beta}(dt)=0
	\end{align*}
	This contradicts the fact that $u$ is in the support of the Parisi measure and completes the proof of our claim.
	Next, note that evidently
	\begin{align*}
	\e C\!M_{N}(u,\varepsilon)\leq \e{C\!F}_{N,\beta}(u,\varepsilon)\leq \frac{\log 4}{\beta}+\e C\!M_{N}(u,\varepsilon).
	\end{align*}
    This together with our claim implies that
    \begin{align*}
    \lim_{N\rightarrow\infty}\e C\!M_{N}(u,\varepsilon)=2\lim_{\beta\rightarrow\infty}F_\beta=2GSE.
    \end{align*}	
From the Gaussian concentration of measure inequality, for any $\eta>0,$ there exists a constant $K'>0$ such that for any $N\geq 1$ with probability at least $1-K'e^{-N/{K'}}$,
    \begin{align*}
    C\!M_{N}(u,\varepsilon)\geq 2GSE-\frac{\eta}{2}
    \end{align*}
    and
    \begin{align*}
    \max_{\sigma\in\Sigma_N}\frac{H_N(\sigma)}{N}\leq GSE+\frac{\eta}{4}.
    \end{align*}
    From the first inequality, there exist $\sigma^1$ and $\sigma^2$ with $|R_{1,2}-u|<\varepsilon$ and
    \begin{align*}
    \frac{1}{N}\bigl(H_N(\sigma^1)+H_N(\sigma^2)\bigr)\geq 2GSE-\frac{\eta}{2}.
    \end{align*}
    If either $H_N(\sigma^1)<GSE-\eta$ or $H_N(\sigma^2)<GSE-\eta$, then we arrive at a contradiction,
    \begin{align*}
 2GSE-\frac{\eta}{2}\leq\frac{1}{N}\bigl(H_N(\sigma^1)+H_N(\sigma^2)\bigr)\leq GSE-\eta+GSE+\frac{\eta}{4}=2GSE-\frac{3\eta}{4}.
    \end{align*}
    Therefore, the inequality \eqref{eq4} must hold and this finishes our proof.	
\end{proof}

\begin{proof}[\bf Proof of Theorem \ref{thm:finiteTemp}]
Recall the Parisi measure $\alpha_{P,\beta}$ for the free energy from \eqref{eq0}. We first claim that $(\beta\alpha_{P,\beta})_{\beta>0}$ converges to $\gamma_P$ vaguely on $[0,1).$ Suppose there exists an infinite sequence $(\beta_l)_{l\geq 1}$ such that $(\beta_l\alpha_{P,\beta_l})_{l\geq 1}$ does not converge to $\gamma_P$ vaguely on $[0,1).$ By an identical argument as \cite[Equation (16)]{AC16}, we can further pass to a subsequence of $(\beta_l\alpha_{P,\beta_l})_{l\geq 1}$ such that it vaguely converges to some $\gamma$ on $[0,1).$ To ease our notation, we use $(\beta_l\alpha_{P,\beta_l})_{l\geq 1}$ to standard for this subsequence. It was established in \cite[Lemma 3]{AC16} that
\begin{align*}
\lim_{l\rightarrow\infty}F_{\beta_l}\geq \mathcal{P}(\gamma).
\end{align*}
From this,
\begin{align*}
\mathcal{P}(\gamma_P)=\lim_{\beta\rightarrow\infty}F_\beta\geq \mathcal{P}(\gamma).
\end{align*}
From  the uniqueness of $\gamma_P$ established in \cite[Theorem 4]{CHL16}, it follows that $\gamma_P=\gamma,$ a contradiction. Thus, $(\beta\alpha_{P,\beta})_{\beta>0}$ converges to $\gamma_P$ vaguely on $[0,1).$ This completes the proof of our claim.

Next, if Theorem \ref{thm:finiteTemp} does not hold, then from the above claim, there exists some $k\geq 1$ such that the support of $\alpha_{P,\beta}$ contains at most $k$ points for all sufficiently large $\beta$. This implies that the support of $\gamma_P$ contains at most $k$ points. This contradicts Theorem \ref{th:main}.
\end{proof}

        

\appendix

\section{Appendix}

Denote by $\Phi(x)$ the c.d.f. of the standard normal distribution. The following lemma follows a standard Gaussian computation and is used in \eqref{add:lem4}.

\begin{lemma}\label{lem1}
	Suppose that $z$ is a standard normal random variable. For any $x,m\in\mathbb{R}$ and $a>0,$
	\begin{align*}
	\e ze^{m|x+az|}\sign(x+az)&=\sqrt{\frac{2}{\pi}}e^{-\frac{x^2}{2a^2}}+ma\e e^{m|x+az|}.
	\end{align*}
\end{lemma}

\begin{proof}
    Define
	$$
	f(y)=e^{m y}\Phi\Bigl(\frac{y}{a}+m a\Bigr),
	$$
	for $y\in\mathbb{R},$ where $\Phi(x)$ is the c.d.f. of the standard Gaussian random variable.
	Note that
	\begin{align*}
	amz-\frac{z^2}{2}&=-\frac{1}{2}(z-m a)^2+\frac{m ^2a^2}{2}.
	\end{align*}
	Computing directly gives
	\begin{align*}
	\e ze^{m (x+az)}\indi_{\{x+az>0\}}&=\frac{e^{m x}}{\sqrt{2\pi}}\int_{-x/a}^{\infty}ze^{am z-\frac{z^2}{2}}dz\\
	&=\frac{e^{m x+\frac{m ^2a^2}{2}}}{\sqrt{2\pi}}\int_{-x/a}^{\infty}(z-m a)e^{-\frac{(z-m a)^2}{2}}dz+\frac{m ae^{m x+\frac{m ^2a^2}{2}}}{\sqrt{2\pi}}\int_{-x/a}^{\infty}e^{-\frac{(z-m a)^2}{2}}dz\\
	&=\frac{e^{-\frac{x^2}{2a^2}}}{\sqrt{2\pi}}+m ae^{\frac{m ^2a^2}{2}}f(x).
	\end{align*}	
	On the other hand, since $z':=-z$ is a standard Gaussian random variable, we may apply the above formula to obtain
	\begin{align*}
	\e ze^{-m (x+az)}\indi_{\{x+az<0\}}&=-\e (-z)e^{m ((-x)+a(-z))}\indi_{\{-x+a(-z)>0\}}\\
	&=-\e z'e^{m ((-x)+az')}\indi_{\{-x+az'>0\}}\\
	&=-\frac{e^{-\frac{x^2}{2a^2}}}{\sqrt{2\pi}}-m ae^{\frac{m ^2a^2}{2}}f(-x).
	\end{align*}
	Combining these two equations together leads to
	\begin{align*}
	\e ze^{m (x+az)}\indi_{\{x+az>0\}}-\e ze^{-m (x+az)}\indi_{\{x+az<0\}}&=\sqrt{\frac{2}{\pi}}e^{-\frac{x^2}{2a^2}}+m ae^{\frac{m ^2a^2}{2}}\bigl(f(x)+f(-x)\bigr).
	\end{align*}
	Here, note that
	\begin{align*}
	\e e^{m |x+az|}&=e^{\frac{m ^2a^2}{2}+xm }\Phi\Bigl(\frac{x}{a}+am \Bigr)+e^{\frac{m ^2a^2}{2}-xm }\Phi\Bigl(-\frac{x}{a}+am \Bigr)=
	e^{\frac{m ^2a^2}{2}}\bigl(f(x)+f(-x)\bigr).
	\end{align*}
	This and the last equation together imply the announced result.
\end{proof}

\end{document}